\documentclass[11pt, reqno]{amsart}
\usepackage{amsmath, amsfonts}
\usepackage{amssymb, graphicx}
\usepackage{amscd}
\usepackage{textcomp}
\usepackage{palatino}
 2

\newtheorem{thm}{Theorem}[section]

\newtheorem{prop}[thm]{Proposition}

\newtheorem{cor}[thm]{Corollary}

\theoremstyle{remark}

\newcommand{\Z}{{\mathbb Z}}

\newcommand{\Q}{{\mathbb Q}}

\begin{document}

\title[On a conjecture of Erd\"{o}s]{On a conjecture of Erd\"{o}s and certain Dirichlet series}

\author[Tapas Chatterjee and M. Ram Murty]{Tapas Chatterjee\textsuperscript{1} and M. Ram Murty\textsuperscript{2}}

\address[T. Chatterjee]
      {Department of Mathematics, 
      Indian Institute of Technology Ropar, 
      Punjab -140001,
      India.}
\address[M. Ram Murty]
       {Department of Mathematics and Statistics,
       Queen's University, Kingston, Ontario,
       Canada, K7L3N6.}
\email[Tapas Chatterjee]{tapasc@iitrpr.ac.in}

\email[M. Ram Murty]{murty@mast.queensu.ca}

\subjclass[2010]{11M06, 11M20}

\keywords{Erd\"os conjecture, non-vanishing of Dirichlet series, Okada's criterion.}

\maketitle

\footnotetext[1]{Research of the first author was supported by a post doctoral fellowship at Queen's University.}

\footnotetext[2]{Research of the second author was supported by an NSERC Discovery grant and a Simons fellowship.}

\begin{abstract}

Let $f:\Z/q\Z\rightarrow\Z$ be such that $f(a)=\pm 1$ for $1\le a<q$, and $f(q)=0$. Then Erd\"{o}s 
conjectured that $\sum_{n\ge1}\frac{f(n)}{n} \ne 0$. For $q$ even, it is easy
to show that the conjecture is true. 
The case $q\equiv 3$ ( mod $4$) was solved by
 Murty and Saradha.  In this paper, we show that this 
conjecture is true for $82\%$ of the remaining integers $q\equiv 1$ ( mod $4$). 

\end{abstract}

\section{\bf Introduction}
In a written communication with Livingston, Erd\"os made the following conjecture (see \cite{AL} ):
if $f$ is a periodic arithmetic function with period $q$ and
$$
f(n) = \begin{cases} 

 \pm 1  & \text{ if $q \nmid n$,}\\

 0  & \text{ otherwise, }

\end{cases}
$$
then 
$$
L(1,f)=\sum_{n= 1}^{\infty}\frac{f(n)}{n} \ne 0
$$
where the $L$-function $L(s,f)$ associated with 
$f$ is defined by the series

\begin{equation}\label{1}
L(s,f):=\sum_{n=1}^\infty \frac{f(n)}{n^s}.
\end{equation}
In 1973, Baker, Birch and Wirsing (see Theorem $1$ of \cite{BBW}), using Baker's theory of linear forms in 
logarithms, proved the conjecture for $q$ prime.
In 1982, Okada \cite{TO1} established the conjecture if $2\varphi(q)+1>q$. Hence, if $q$ is a 
prime power or a product of two distinct odd primes, the conjecture is true. In 2002, R. Tijdeman \cite{RT} proved 
the conjecture is true for periodic completely multiplicative functions $f$.
Saradha and Tijdeman \cite{ST} showed that if $f$ is 
periodic and multiplicative with $| f ( p^k)| < p-1$ for every prime divisor  $p$ of $q$ and every positive integer $k$, then the conjecture is true.
\par
 It is easy to see that 
\begin{equation*}
L(1,f)= \sum_{n =1}^{\infty}\frac{f(n)}{n} 
\end{equation*}
exists if and only if $\sum_{n= 1}^{q}f(n)=0$. If $q$ is even and $f$ takes values $\pm 1$ with $f(q)=0$, then 
$\sum_{n = 1}^{q}f(n)\ne 0$. Hence the conjecture holds for even $q$.
\par
In 2007, Murty and Saradha \cite{MS1} proved that if $q$ is odd and $f$ is an odd integer 
valued odd periodic function  then the conclusion of the conjecture holds. In 2010, they 
proved that the Erd\"os conjecture is true if $q\equiv 3$ ( mod $4$) 
(see Theorem $7$ of \cite{MS}). Thus the conjecture is open only in cases where 
$q\equiv 1$ (mod $4$). However, it seems that a novel idea will be needed to deal with these cases. 
In this paper, we adopt a new density-theoretic approach which is orthogonal to earlier methods. 
Here is the main consequence of our method:
\par
\begin{thm}
\sl Let $S(X)=|\{ q\equiv 1 ~(~{\rm mod} ~4), ~q\le X |$ Erd\"{o}s conjecture is true for $q \}|$. \\
Then $$\underset{X \rightarrow \infty}{\liminf} \frac{S(X)}{X/4}\ge 0.82.$$ 
\end{thm}
\par
In other words, the Erd\"{o}s conjecture is true for at least 82 \% of the integers $q \equiv 1 $ (mod 4).  
Our method does not extend to show that the Erd\"{o}s conjecture is true for 100 \% of the moduli $q\equiv 1$ (mod $4$). 
We examine this question briefly at the end of the paper. It seems to us that more ideas are needed to resolve the 
conjecture fully. 
\par
These questions have a long history beginning with Baker, Birch and Wirsing \cite{BBW}. Their work was generalized 
by Gun, Murty and Rath \cite{GMR} to the setting of algebraic number fields. The forthcoming paper \cite{CM} gives new proofs of some of the background results of
this area. We also refer the reader to \cite{RT} for an expanded survey of the early history.

\section{\bf Notations and Preliminaries}
From now onwards, we denote the field of rationals by $\Q$, the field of algebraic numbers by $\overline{\Q}$,
Euler's totient function by $\varphi$ and Euler's constant by $\gamma$.
We say a function $f$ is Erd\"{o}sian mod $q$ if $f$ is a periodic function with period $q$ and
$$
f(n) = \begin{cases} 

 \pm 1  & \text{ if $q \nmid n$,}\\

 0  & \text{ otherwise.}

\end{cases}
$$
Also we will write $f(X)\lesssim g(X)$ to mean 

$$\underset{X \rightarrow \infty}{\limsup} \frac{f(X)}{g(X)}\le 1.$$
Similarly, we write $f(x)\gtrsim g(x)$ to mean

$$\underset{X \rightarrow \infty}{\liminf} \frac{f(X)}{g(X)}\ge 1.$$

\subsection{\bf Okada's criterion}
\par
\begin{prop}
 Let the $q$-th cyclotomic polynomial $\Phi_q$ be irreducible over the field $\Q(f(1),\cdots, f (q))$. 
Let $M(q)$ be the set of positive integers which are composed of prime factors of $q$.

Then $L(1,f)=0$ if and only if 

\begin{equation*}
 \sum_{m\in M(q)}{f(am)\over m}=0
\end{equation*}
for every $a$ with $1\le a <q, (a,q)=1$, and 

\begin{equation*}
 \sum_{\substack{r=1\\ (r,q)>1}}^qf(r)\epsilon(r,p)=0
\end{equation*}
for every prime divisor $p$ of $q$,  where
$$
 \epsilon(r,p)=\begin{cases} 
v_p(r)  & \text{ if $v_p(r)<v_p(q)$,}\\

v_p(q)+\frac{1}{p-1}  & \text{ otherwise }
\end{cases}
$$
and for any integer $r$, $v_p(r)$ is the exponent of $p$ dividing $r$.
\end{prop}
\par
This Proposition is a modification, due to Saradha and Tijdeman \cite{ST}, of a 1986 result 
of Okada \cite{TO}. Note that Okada deduced the sufficient condition
$2\varphi(q)+1>q$ stated in the introduction from his original version of this criterion.
\subsection{Wirsing's Theorem}
The following proposition is due to Wirsing \cite{EW}. 
\par
\begin{prop}\label{wirsing}
\sl Let $f$ be a non-negative multiplicative arithmetic function, 
satisfying

\begin{equation*}
 |f(p)|\le G ~{\rm for ~all ~primes} ~p,
\end{equation*}

\begin{equation*}
 \sum_{p\le X}p^{-1}f(p)\log p\sim \tau\log X,
\end{equation*}
with some constants $G>0$, $\tau>0$ and

\begin{equation*}
 \sum_{p}\sum_{\substack{k\ge 2}}p^{-k}|f(p^k)|<\infty;
\end{equation*}
if $0<\tau\le 1$, then, in addition, the condition

\begin{equation*}
 \sum_{p}\sum_{\substack{k\ge 2 \\ p^k\le X}}|f(p^k)|=O(X/\log X)
\end{equation*}
is assumed to hold. Then

\begin{equation*}
 \sum_{n\le X}f(n)=(1+o(1))\frac{X}{\log X}\frac{e^{-\gamma\tau}}{\Gamma(\tau)}
 \prod_{p\le X}(1+\frac{f(p)}{p}+\frac{f(p^2)}{p^2}+\cdots).
\end{equation*}
\end{prop}
\subsection{\bf Mertens Theorem}
We also need a classical theorem of Mertens in a later section.  We record the 
theorem here (see for example, p. 130 of \cite{murty}):
\par
\begin{prop}
 $\underset{X\rightarrow \infty}{\lim} \log X\underset{p\le X}{\prod}(1-\frac{1}{p})=e^{-\gamma}. $
\end{prop}
\section{\bf Exceptions to the  conjecture of Erd\"{o}s}
We say that the Erd\"{o}s conjecture is false (mod $q$), if there is an
Erd\"{o}sian  function $f$ for which $L(1,f)=0$. The following proposition
plays a fundamental role in our approach.
\par
\begin{prop}\label{crit}
\sl If the Erd\"{o}s conjecture is false (mod $q$) with
$q$ odd, then
\begin{equation*}
 1\le \sum_{\substack{d|q \\d\ge3}}\frac{1}{\varphi(d)}.
\end{equation*}
\end{prop}
\begin{proof}

 By the hypothesis, there is an Erd\"{o}sian function $f$ mod $q$ for which, we have $L(1,f)=0$.
 Applying Okada's criterion, we get

 \begin{equation}\label{11}
  \sum_{b\in M(q)}\frac{f(b)}{b}=0.
 \end{equation}

Let $d=(b,q)$, so that $b=db_1$ with $(b_1, q/d)=1$. Then \eqref{11} can be written as 

\begin{equation*}
 -f(1)=\sum_{\substack{d|q \\ d\ge3}}\frac{1}{d}\sum_{\substack{b_1\in M(q) \\ (b_1,q/d)=1}}\frac{f(db_1)}{b_1}.
\end{equation*}

Taking absolute value of both sides, we get

\begin{equation}\label{12}
 1\le \sum_{\substack{d|q \\ d\ge3}}\frac{1}{d}\sum_{\substack{b_1\in M(d)}}\frac{1}{b_1}.
\end{equation}

 Notice that the inner sum can be written as

$$
 \sum_{\substack{b_1\in M(d)}}\frac{1}{b_1}
 = \prod_{p|d}\left(1+\frac{1}{p}+\frac{1}{p^2}+\cdots \right) 
 =\prod_{p|d}\left(1-\frac{1}{p}\right)^{-1} 
 =\frac{d}{\varphi(d)}. $$
Hence from \eqref{12}, we get

\begin{equation*}
 1\le \sum_{\substack{d|q \\ d\ge3}}\frac{1}{\varphi(d)}.
\end{equation*} 

\end{proof}
\par
 Two immediate corollaries of the above proposition are the following:
\par
\begin{cor}
\sl If $q$ is a prime power or a product of two distinct odd primes, 
then the Erd\"{o}s conjecture is true (mod $q$).
\end{cor}
\begin{proof}
This is a pleasant elementary exercise.
\end{proof}
\par
Hence we have recovered the two basic cases of the conjecture which were given 
in the introduction, of course, also as a consequence of Okada's
Criterion.
\par
Let $d(n)$ be the divisor function, i.e. $d(n)$ is the number of divisors of $n$. 

\par
\begin{cor}\label{new}
\sl If the smallest prime factor of $q$ is at least $d(q)$, then 
the Erd\"{o}s conjecture is true for $q$.
\end{cor}
\begin{proof}
 Let $l$ be the smallest prime factor of $q$. From the above proposition, if the Erd\"{o}s conjecture
 is false (mod $q$), then we have 
 \begin{eqnarray*}
 1&\le& \sum_{\substack{d|q \\ d\ge3}}\frac{1}{\varphi(d)} \\
 & < & \frac{1}{\varphi(l)}\sum_{\substack{d|q \\ d\ge3}} 1= \frac{d(q)-2}{l-1},
 \end{eqnarray*}
the strict inequality in the penultimate step coming from the fact
that $q$ has at least two prime divisors.  Thus, 
 $l <  d(q)$. Hence if $l\ge d(q)$, then the Erd\"{o}s conjecture is true (mod $q$).

\end{proof}
\par
Note that, the Corollary \ref{new} was not known 
previously.  It implies that the conjecture is true for any squarefree
number $q$ with $k$ prime factors, provided the smallest prime factor of $q$ is 
greater than $2^k$.  
Proposition \ref{crit} opens the door for a new approach to the study
of Erd\"os's conjecture.  
Let us  consider the following 
$$S_1(X)=|\{ q\equiv 1 ~(~{\rm mod} ~4), ~q\le X | \text{Erd\"{o}s conjecture is false (mod $q$)} \}|$$ 
Then, we have 
\begin{eqnarray*}
 S_1(X)&\le& \sum_{\substack{q\le X \\ q\equiv 1 ~(~{\rm mod} ~4)}}\sum_{\substack{d|q \\ d\ge3}}\frac{1}{\varphi(d)} 
 \le \sum_{\substack{3\le d\le X\\ d ~{\rm odd}}}\frac{1}{\varphi(d)}\sum_{\substack{q\le X \\ q\equiv 1 ~(~{\rm mod} ~4) \\ d|q}} 1 \\
 &\le& \sum_{\substack{3\le d\le X\\ d ~{\rm odd}}}\frac{1}{\varphi(d)} \left(\frac{X}{4d}+O(1)\right) 
\le \sum_{\substack{3\le d\le X\\ d ~{\rm odd}}}\frac{1}{\varphi(d)}\frac{X}{4d}+
 O\left(\sum_{\substack{3\le d\le X}}\frac{1}{\varphi(d)}\right)\\
&\le& \sum_{\substack{3\le d\le X\\ d ~{\rm odd}}}\frac{1}{\varphi(d)}\frac{X}{4d}+O(\log X) 
 \end{eqnarray*}
where we have used the well-known fact that (see for example, p. 67 of \cite{murty})

\begin{equation*}
\sum_{\substack{d\le X}}\frac{1}{\varphi(d)}= O(\log X).
\end{equation*}
 Hence, we get
 \begin{eqnarray*}
 S_1(X)&\lesssim& \frac{X}{4}\sum_{\substack{3\le d\\ d ~{\rm odd}}}\frac{1}{d\varphi(d)}\\
&\lesssim& \frac{X}{4} \left(\prod_{p ~{\rm odd}}\left(1+\frac{1}{p\varphi(p)}+\frac{1}{p^2\varphi(p^2)}+\cdots\right)-1\right)\\
&\lesssim&\frac{X}{4} \left(\prod_{p ~{\rm odd}}\left(1+\frac{1}{p(p-1)}+\frac{1}{p^3(p-1)}+\cdots\right)-1\right)\\
&\lesssim&\frac{X}{4} \left(\prod_{p ~{\rm odd}}\left(1+\frac{1}{p(p-1)}\left(1+{1\over p^2}+{1\over p^4}+\cdots\right)\right)-1\right)\\
&\lesssim&\frac{X}{4} \left(\prod_{p ~{\rm odd}}\left(1+\frac{p}{(p-1)(p^2-1)}\right)-1\right).
\end{eqnarray*}
The product is easily computed numerically and we have:
$
S_1(X)\lesssim 0.33( {X/4}).
$
As an immediate corollary we get the following:
\begin{cor}
 $|\{ q\equiv 1 ~(~{\rm mod} ~4), ~q\le X |$ Erd\"{o}s conjecture is true for $q \}|$ \\
$\gtrsim 0.67\frac{X}{4}$.
\end{cor}
\subsection{Refinement using the second moment}
By considering higher moments, we can improve the lower bound in the above corollary. 
We begin with the second moment. We include these estimates since they are of independent 
interest and self-contained.
\par
\begin{prop}
 $|\{ q\equiv 1 ~(~{\rm mod} ~4), ~q\le X |$ Erd\"{o}s conjecture is true for $q \}|$ \\
$\gtrsim 0.78\frac{X}{4}$. 
\end{prop}
\begin{proof}
Let us first consider the following inequality:
\begin{eqnarray*}
 S_1(X)&\le& \sum_{\substack{q\le X \\ q\equiv 1 ~(~{\rm mod} ~4)}}\left(\sum_{\substack{d|q \\ d\ge3}}\frac{1}{\varphi(d)}\right)^2 \\
&\le& \sum_{\substack{q\le X \\ q\equiv 1 ~(~{\rm mod} ~4)}}\sum_{\substack{d_1|q, ~d_2|q \\ 3\le d_1,~d_2 <q}}
\frac{1}{\varphi(d_1)\varphi(d_2)} \\ 
&\le& \sum_{\substack{3\le d_1,~d_2\le X \\ d_1,~d_2 ~{\rm odd}}}\frac{1}{\varphi(d_1)\varphi(d_2)} 
\sum_{\substack{q\le X \\ q\equiv 1 ~(~{\rm mod} ~4) \\ d_1|q,~d_2|q}} 1 \\
&\le& \sum_{\substack{3\le d_1,~d_2\le X \\ d_1,~d_2 ~{\rm odd}}}\frac{1}{\varphi(d_1)\varphi(d_2)} 
\sum_{\substack{q\le X \\ q\equiv 1 ~(~{\rm mod} ~4) \\ [d_1,d_2]|q}} 1 \\
&\le& \sum_{\substack{3\le d_1,~d_2\le X \\ d_1,~d_2 ~{\rm odd}}}\frac{1}{\varphi(d_1)\varphi(d_2)} 
\left(\frac{X}{4[d_1,d_2]}+O(1)\right).
 \end{eqnarray*}
 Hence, we have
\begin{eqnarray*}
S_1(X)&\le& \frac{X}{4}\sum_{\substack{3\le d_1,~d_2\le X \\ d_1,~d_2 ~{\rm odd}}}\frac{1}{\varphi(d_1)
\varphi(d_2)[d_1,d_2]}+O(\log^2 X).
 \end{eqnarray*}
 By a simple numerical calculation, we deduce that
 \begin{equation*}
S_1(X)\lesssim 0.22 \frac{X}{4}.  
 \end{equation*}
 Hence the conjecture holds for at least 78\% of the positive integers congruent to 1 mod 4.
\end{proof}
\par
Similarly one can compute higher fractional moments to get an optimal result. For any $r>1$, we have
\begin{equation*}
 S_1(X)\le\sum_{\substack{q\le X \\ q\equiv 1 ~(~{\rm mod} ~4)}}\left(\sum_{\substack{d|q \\ d\ge3}}\frac{1}{\varphi(d)}\right)^r.
\end{equation*}
We study this as a function of $r$.  
Using Maple we computed that the minimal value occurs at $r\sim3.85$ (see 
\cite{code}) and we get
\begin{equation*}
S_1(X)\lesssim 0.18 \frac{X}{4}.  
 \end{equation*}
Thus, we get $|\{ q\equiv 1 ~(~{\rm mod} ~4), ~q\le X |$ Erd\"{o}s conjecture is true for $q \}|$ \\
$\gtrsim 0.82\frac{X}{4}$, i.e.
$$\underset{X \rightarrow \infty}{\liminf} \frac{S(X)}{X/4}\ge 0.82.$$
Hence, we have shown the theorem stated in the Introduction, i.e. the conjecture 
holds for at least 82\% of the numbers congruent to 1 mod 4. 
\par
\subsection{An alternative approach}
In this subsection, we discuss an alternative approach to this problem. It leads to a slightly weaker result. 
However this method is of independent interest, so we record it here. We begin with a further refinement of  
Proposition \ref{crit} by considering fractional moments there. From Proposition \ref{crit}, 
we get if the Erd\"{o}s conjecture is false for odd $q$, then
\begin{equation*}
 1\le \sum_{\substack{d|q \\ d\ge3}}\frac{1}{\varphi(d)}.
\end{equation*}
Adding 1 both sides of the above inequality, we get 
\begin{equation*}
 2\le \sum_{\substack{d|q }}\frac{1}{\varphi(d)},
\end{equation*}
which can be rewritten as
\begin{equation*}
 1\le \frac{1}{2}\sum_{\substack{d|q }}\frac{1}{\varphi(d)}.
\end{equation*}
Hence for any $\alpha>0$, Proposition \ref{crit} can be rewritten as:
\par
\begin{prop}
 If Erd\"{o}s conjecture is false for odd $q$, then
\begin{equation*}
 1\le \frac{1}{2^\alpha}\left(\sum_{\substack{d|q }}\frac{1}{\varphi(d)}\right)^\alpha.
\end{equation*}
\end{prop}
As before, $S_1(X)=|\{ q\equiv 1 ~(~{\rm mod} ~4), ~q\le X |$ Erd\"{o}s conjecture is false for $q \}|$. 
Then from the above proposition, we get

\begin{equation*}
 S_1(X)\le \frac{1}{2^\alpha}\sum_{\substack{q\le X \\ q\equiv 1 ~(~{\rm mod} ~4)}}\left(\sum_{\substack{d|q }}
 \frac{1}{\varphi(d)}\right)^\alpha. 
\end{equation*}

Let $f_\alpha(q)=\left(\sum_{\substack{d|q }}\frac{1}{\varphi(d)}\right)^\alpha$ and $\chi$ be the 
non-trivial Dirichlet character mod 4. Then the above inequality becomes
\begin{equation}\label{13a}
 S_1(X)\le \frac{1}{2^{\alpha+1}}\left(\sum_{\substack{q\le X \\ q ~{\rm odd}}}f_\alpha(q)+
 \sum_{\substack{q\le X \\ q ~{\rm odd}}}\chi(q)f_\alpha(q)\right). 
\end{equation}
Again, note that $f_\alpha(q)$ is a multiplicative arithmetic function. One can check that it also satisfies 
all the other hypotheses of Wirsing's theorem (Proposition \ref{wirsing}) with $G=2^\alpha$ and 
$\tau=1$. So in light of Wirsing's theorem, we get 
\begin{equation*}
 \sum_{\substack{q\le X \\ q ~{\rm odd}}}f_\alpha(q)\sim X\frac{e^{-\gamma}}{\log X}
 \prod_{\substack{p\le X \\ p\ne 2}}\left(1+\frac{f_\alpha(p)}{p}+\frac{f_\alpha(p^2)}{p^2}+\cdots\right)
\end{equation*}
and 
\begin{equation*}
 \sum_{\substack{q\le X \\ q ~{\rm odd}}}\chi(q)f_\alpha(q)\sim X\frac{e^{-\gamma}}{\log X}
 \prod_{\substack{p\le X \\ p\ne 2}}\left(1+\frac{\chi(p)f_\alpha(p)}{p}+\frac{\chi(p^2)f_\alpha(p^2)}{p^2}+\cdots\right).
\end{equation*}
Again, from Mertens theorem we know that 
\begin{equation*}
 \prod_{p\le X}(1-1/p)\sim \frac{e^{-\gamma}}{\log X}. 
\end{equation*}
Hence we have
\begin{eqnarray*}
 \sum_{\substack{q\le X \\ q ~{\rm odd}}}f_\alpha(q)&\sim& \frac{X}{2} \prod_{\substack{p\le X \\ p\ne 2}}
 (1-1/p)\left(1+\frac{f_\alpha(p)}{p}+\frac{f_\alpha(p^2)}{p^2}+\cdots\right) \\ 
 &\sim& \frac{X}{2} p_1 ~({\rm say}) 
\end{eqnarray*}
and 
\begin{eqnarray*}
 \sum_{\substack{q\le X \\ q ~{\rm odd}}}\chi(q)f_\alpha(q)&\sim& \frac{X}{2} \prod_{\substack{p\le X \\ p\ne 2}}
 (1-1/p)\left(1+\frac{\chi(p)f_\alpha(p)}{p}+\frac{\chi(p^2)f_\alpha(p^2)}{p^2}+\cdots\right)\\
  &\sim& \frac{X}{2} p_2 ~({\rm say}).
\end{eqnarray*}
Now using the above two inequality \eqref{13a} becomes,
\begin{equation*}
 S_1(X)\lesssim \frac{X}{2^{\alpha+2}}(p_1+p_2).
\end{equation*}
Finally using Maple (see \cite{code}), we find that the quantity on the right hand side
is minimized at $\alpha\sim 8.11$ and we get 
\begin{equation*}
 S_1(X)\lesssim 0.20\frac{X}{4}.
\end{equation*}
Hence, we get  $$\underset{X \rightarrow \infty}{\liminf} \frac{S(X)}{X/4}\ge 0.80.$$
\par
\noindent
\textbf{Remarks}. One cannot hope to obtain 100 \% by these methods.  In fact, one can
show that there is a positive density (albeit small)
 of $q$ for which the inequality of Proposition \ref{crit} holds.  Indeed, since
$$\sum_{d|q} {1 \over \varphi(d)} \geq \prod_{p|q} \left( 1 + {1\over p-1}
\right) $$
we can make the product ( and hence the sum) arbitrarily large by ensuring that $q$ is divisible
by all the primes in an initial segment.  We can even ensure that
these primes are congruent to 1 (mod 4).  We then take numbers
which are divisible by this $q$ and congruent to 1 (mod 4) and deduce
that for all these numbers, the inequality in the proposition holds.  
Since the product on the right diverges slowly to infinity as we go through such 
numbers $q$, we obtain in this way, a small density of numbers for which the inequality
holds.  
\par
\noindent
{\bf Acknowledgements.}
We thank Michael Roth for his help on using the Maple language
as well as Sanoli Gun, Purusottam Rath and Ekata Saha for their comments on an
earlier version of this paper. We also thank the referee for helpful comments 
that improved the quality of the paper.

\end{document}